\documentclass[12pt]{article}
\usepackage[a4paper, margin=2.5cm]{geometry}

\usepackage{amsmath}
\usepackage{amssymb}
\usepackage{amsthm}
\usepackage{mathtools}
\usepackage{amsfonts}
\usepackage{tikz-cd}
\usepackage{enumerate}
\usepackage[utf8]{inputenc}
\usepackage[T1]{fontenc}
\usepackage{siunitx}

\usepackage{url}
\usepackage{hyperref}
\usepackage{blindtext}
\usepackage{cleveref}

\usepackage{mathrsfs}
\usepackage{bbm}
\usepackage{bbold}
\usepackage[mathscr]{euscript}
\usepackage{stmaryrd}
\usepackage{indentfirst}
\usepackage{libertine}
\usepackage{calligra}
\usepackage{calrsfs}

\usepackage{appendix}

\newcommand{\basef}{\mathbb{k}}

\newcommand{\base}{\mathrm{S}}
\newcommand{\baseS}{\mathrm{S}}

\newcommand{\Spec}{\operatorname{Spec}}
\newcommand{\ZZ}{\mathbb{Z}}
\newcommand{\QQ}{\mathbb{Q}}
\newcommand{\FF}{\mathbb{F}}
\newcommand{\ring}{\Lambda}

\newcommand{\X}{\mathscr{X}}
\newcommand{\Y}{\mathscr{Y}}
\newcommand{\Z}{\mathscr{Z}}

\newcommand{\U}{\mathscr{U}}

\newcommand{\sX}{\mathrm{X}}
\newcommand{\sY}{\mathrm{Y}}
\newcommand{\sZ}{\mathrm{Z}}
\newcommand{\sU}{\mathrm{U}}

\newcommand{\GL}{\mathrm{GL}}
\newcommand{\G}{\mathrm{G}}
\renewcommand{\L}{\mathrm{L}}
\renewcommand{\P}{\mathrm{P}}
\newcommand{\M}{\mathrm{M}}

\newcommand{\Ct}{\mathrm{Z}}
\newcommand{\T}{\mathrm{T}}
\newcommand{\sW}{\mathrm{W}}
\newcommand{\Nm}{\mathrm{N}}

\newcommand{\GG}{\mathbb{G}}
\newcommand{\Bs}{\mathrm{B}}

\newcommand{\N}{\mathscr{N}}

\renewcommand{\AA}{\mathbb{A}}
\newcommand{\Or}{\mathscr{O}}

\newcommand{\ind}{\mathscr{I}}

\newcommand{\St}{\mathrm{St}}
\newcommand{\Fl}{\operatorname{Fl}}

\newcommand{\Hom}{\operatorname{Hom}}

\newcommand{\iHom}{\underline{\operatorname{Hom}}}

\newcommand{\End}{\operatorname{End}}
\newcommand{\id}{\mathrm{id}}

\newcommand{\R}{\mathrm{R}}
\renewcommand{\H}{\mathrm{H}}

\newcommand{\CH}{\operatorname{CH}}
\newcommand{\Tilt}{\mathscr{T}}
\newcommand{\Corr}{\mathrm{Corr}}
\newcommand{\Sym}{\operatorname{Sym}}

\newcommand{\Comp}{\operatorname{Comp}}
\newcommand{\Compf}{\operatorname{Compf}}

\newcommand{\Mod}{\mathrm{Mod}}
\newcommand{\grMod}{\mathrm{grMod}}

\newcommand{\Sch}{\mathrm{Sch}}

\newcommand{\Rep}{\operatorname{Rep}}

\newcommand{\A}{\mathscr{A}}

\newcommand{\C}{\mathscr{C}}

\renewcommand{\Pr}{\mathrm{Pr}}
\newcommand{\Sp}{\mathsf{Sp}}

\newcommand{\D}{\operatorname{D}}

\newcommand{\h}{\mathrm{h}}
\newcommand{\Ch}{\operatorname{Ch}}
\newcommand{\K}{\mathrm{K}}

\renewcommand{\L}{\mathrm{L}}
\newcommand{\st}{\mathrm{st}}

\newcommand{\red}{\mathrm{red}}
\newcommand{\gm}{\mathrm{gm}}

\newcommand{\m}{\mathrm{m}}

\newcommand{\op}{\mathrm{op}}

\newcommand{\add}{\mathrm{add}}

\newcommand{\Spr}{\mathrm{Spr}}

\newcommand{\dv}{\mathbf{d}}
\newcommand{\pdv}{\underline{\mathbf{d}}}

\newcommand{\redm}{\mathrm{r}}

\let\svc\c
\DeclareRobustCommand\c{\ifmmode{c}\else\expandafter\svc\fi}

\newcommand{\DTM}{\operatorname{DTM}}

\newcommand{\DM}{\operatorname{DM}}

\DeclareMathOperator{\Fun}{Fun}
\DeclareMathOperator{\Map}{Map}

\newtheorem{prop}{Proposition}[section]
\newtheorem{lemma}[prop]{Lemma}
\newtheorem{thm}[prop]{Theorem}
\newtheorem{cor}[prop]{Corollary}

\newtheorem{theorem}{Theorem}

\theoremstyle{definition}
\newtheorem{defn}[prop]{Definition}
\newtheorem{example}[prop]{Example}

\theoremstyle{remark}
\newtheorem{rem}[prop]{Remark}

\title{An Integral Springer Correspondence}
\author{Thiago Landim}

\begin{document}

\maketitle

\begin{abstract}
    We prove an integral version of the derived Springer correspondence for reduced motives. To achieve this result, we extend some results on reduced motives from schemes to quotient stacks with a finite number of orbits. More generally, we work in the context of the Springer setup, as defined by Eberhardt and Stroppel, which has also applications for quiver Hecke algebras.
\end{abstract}

\tableofcontents

\section*{Introduction}

Let $\G$ be a reductive group with Weyl group $\sW$, let $\N_\G$ be its nilpotent cone. The Steinberg variety $\Z \coloneqq \widetilde{\N}_\G \times_{\N_\G} \widetilde{\N}_\G$ is given by the fiber product of the Springer resolution with itself. The Springer correspondence, discovered by Springer in \cite{springer_1}, classically gives a correspondence between irreducible $\overline{\QQ}_\ell$-representations of $\sW$ and some irreducible $\ell$-adic perverse sheaves on $\N_\G$. 

Over time, three directions to generalize such a correspondence were researched. First of all, Lusztig found in \cite{int_coh} what is now called the generalized Springer correspondence, which para-metrizes \emph{all} $\G$-equivariant irreducible perverse sheaves on $\N_\G$ using the \emph{relative} Weyl groups of $\G$. Secondly, Achar et al. in \cite{weyl_fourier_springer_integral} proved that the original correspondence existed already with coefficients being a commutative Noetherian ring\footnote{Springer had already noticed in \cite{springer_2} that one may use classical cohomology to find the correspondence over $\QQ$.} (of finite global dimension). Finally, Rider in \cite{rider_springer} proved a derived version of the classical Springer  correspondence.

Recently, Eberhardt and Stroppel in \cite{springer_motives} and \cite{motivic_springer} developed a motivic Springer correspondence with rational coefficients. This extends Rider's work, and shows, in particular, that the derived correspondence is independent of the prime $\ell$.

Eberhardt's and Stroppel's approach uses ideas developed by Soergel, Virk and Wendt in \cite{perverse_category_o} and \cite{eq_mot} which lays down the foundation of equivariant motives with rational coefficients. Their computation relies on the fact that the higher $\K$-theory of $\FF_q$ is all torsion and are therefore trivialized by taking rational coefficients.

In a different direction, Eberhardt and Scholbach in \cite{reduced_motives} defined the notion of a reduced motive. For a finite type scheme $\sX$ over a base $\base$, the category of reduced motives of $\sX$ are obtained from the category of motives of $\sX$ after killing the information coming from the motivic cohomology of the base $\base$. In particular, by the Atiyah-Hirzebruch spectral sequence, this should allow one to work as if the higher $\K$-theory of the base $\base$ was trivial. We are going to build on Eberhardt's and Stroppel's ideas in the context of reduced motives.

To prove the integral Springer correspondence, we extend the whole Springer setup. In particular, this allows us to apply the same formalism to quiver Hecke algebras and quiver Schur algebras, as in \cite{motivic_springer}. Before stating the results, we will fix some notation.

Let $\basef$ be a field, let $\G$ be a linear algebraic group over $\basef$, let $\mu_i \colon \widetilde{\N}_i \to \N$ be a collection of $\G$-equivariant proper morphisms between quasi-projective varieties with $\G$-actions, let $\ring$ be a commutative ring of coefficients and consider, as in \cite{motivic_springer}, the \emph{motivic extension algebra}
\[
    \mathrm{E} = \bigoplus_{n \in \ZZ} \bigoplus_{i,j} \Hom_{\DM(\N/\G, \ring)}(\mu_{i!}\ring, \mu_{j!}(\ring)[2n](n)).
\] 

For schemes, it is a classical result that such $\Hom$-groups may be computed as Chow groups of the fiber product. For algebraic stacks over a field, this is done, for example, by Aranha and Chowdhury in \cite[Theorem 5.2]{Chow-weight-stacks}. Denoting $\Z_{i,j} \coloneqq \widetilde{\N}_i \times_\N \widetilde{\N}_j$, one may describe the algebra $\mathrm{E}$ as
\[
    \mathrm{E} \cong \bigoplus_{i, j} \CH_{\bullet - d_j}^\G(\Z_{i,j})_\ring,
\]
where $d_j$ is the dimension of $\widetilde{\N}_j$.

Under this setup, one defines the category $\DM_\redm^\Spr(\N/\G, \ring)$ of \emph{reduced Springer motives} as the thick subcategory of $\DM_\redm(\N/\G, \ring)$ generated by Tate twists of the objects $\mu_{i!}\ring$. After introducing purity, finiteness and niceness conditions, which we call (PT), (FO) and (RG) (see \S 2.1), we prove the following.

\begin{theorem}[Theorem \ref{thm:springer_setup}]
    Under the hypotheses (PT), (FO) and (RG), there exists an equivalence of $\infty$-categories
    \[
        \DM_\redm^\Spr(\N/\G, \ring) \cong \mathrm{grPerf}(\mathrm{E}),
    \]
    where $\mathrm{grPerf}(\mathrm{E})$ denotes the category of graded perfect complexes over $E$. If $\basef = \overline{\FF}_p$ and $\ring = \QQ$, this equivalence is the same as in \cite[Theorem 4.9]{motivic_springer}.
\end{theorem}

As promised, we apply the formalism above in the context of a single map, the Springer resolution $\mu \colon \widetilde{\N}_\G \to \N_\G$. Let $\basef$ be an algebraically closed field with characteristic exponent $p$, and let $\Lambda$ be a $\ZZ[\tfrac{1}{p}]$-algebra. Let $\G$ be a reductive group over $\basef$, let $\N_\G$ be the nilpotent cone of $\G$ and $\Z \coloneqq \widetilde{\N}_\G \times_{\N_\G} \widetilde{\N}_\G$ the Steinberg variety.

\begin{theorem}[Theorem \ref{thm:springer_correspondence}]
    Under rather mild hypotheses on $p$ and $\ring$, there exists an equivalence of $\infty$-categories
    \[
        \DM_{\mathrm{r}}^{\Spr}(\N_\G/\G, \Lambda) \cong \mathrm{grPerf}(\CH_\bullet^\G(\Z)_\ring).
    \]
    Moreover, this equivalence commutes with parabolic induction.
\end{theorem}

This generalizes both the rational motivic Springer theory in \cite[Theorem 5.2]{motivic_springer} and, at least for $\GL_n$, the integral Springer correspondence given in \cite[Corollary 5.3]{weyl_fourier_springer_integral}. Applications to quiver Hecke algebras, as done by \cite[\S 6]{motivic_springer}, are sketched in \S 2.3.

\subsection*{Notations}

We denote by $\Pr_\st^\L$ the category of presentable stable $\infty$-categories and functors between them preserving colimits. If $\C$ is a stable $\infty$-category and $x, \, y \in \C$, we will denote by $\Map_\C(x,y)$ the mapping space, by $\operatorname{map}_\C(x,y)$ the mapping spectra and by $\Hom_\C(x,y) \coloneqq \pi_0(\Map_\C(x,y))$ the Hom-set. If $\C$ is a presentable stable category endowed with a fixed automorphism $\langle 1 \rangle$ and if $c \in \C$, we will follow \cite{motivic_springer} and denote $\End^\bullet(c) \coloneqq \bigoplus_n \Hom_{\C}(c, c[n]\langle n \rangle)$ with the natural composition.

\section{Reduced motives on quotient stacks}

\subsection{Reduced motives and Tate motives}

We are going to review some definitions and properties of (reduced) motives on stacks. The results on schemes are all contained in \cite{reduced_motives}. The main new step is the computation of some $\Hom$-groups on the category of reduced motives on a classifying space. 

First of all, let $\baseS$ be a connected smooth finite type scheme over a Dedekind ring and let $\ring$ be a regular coherent ring. For $\sX$ a finite type scheme over $\baseS$, we are going to denote by $\DM(\sX, \ring)$ the category of \emph{motives over $\sX$} as defined by Spitzweck in \cite{spitzweck}. Since we are working exclusively with finite type schemes, we will denote by $\Sch_\baseS$ the category of finite type schemes over $\baseS$. Moreover, if $\ring$ is already clear from context, we will denote $\DM(\sX, \ring)$ simply by $\DM(\sX)$. Before going any further, we recall some basic properties of $\DM$ and extend them to (algebraic) stacks. 

\begin{prop}[\S 10, \cite{spitzweck}] \label{prop:six_functors_DM}
    The categories of motives satisfy the following functorial properties:
    \begin{enumerate}[(1)]
        \item For any morphism $f \colon \sY \to \sX$, there exists an adjunction $f^* \colon \DM(\sX,\ring) \leftrightarrows \DM(\sY, \ring) \colon f_*$;
        \item For any morphism $f \colon \sY \to \sX$, there exists an adjunction $f_! \colon \DM(\sX,\ring) \leftrightarrows \DM(\sY, \ring) \colon f^!$;
        \item For any $\sX \in \Sch_{\baseS}$, the category $\DM(\sX)$ is closed symmetric monoidal;
        \item \emph{Nisnevich descent:} The functor $\DM^* \colon \Sch_{\baseS}^{\op} \to \Pr^\L_\st$ is a Nisnevich sheaf.
        \item \emph{Base change formulas:} For any Cartesian diagram
        \[
        \begin{tikzcd}
            \sY' \ar["g'"']{d} \ar["f'"]{r} & \sX' \ar["g"]{d} \\
            \sY \ar["f"']{r} & \sX,
        \end{tikzcd}
        \]
        there exists a natural isomorphism $g_!f'^* \to f^*g'_!$. Dually, there is also a natural isomorphism $g^!f'_* \to f_* g'^!$.
        \item \emph{Projection formulas:} For any $f \colon \sY \to \sX$ and for any motives $M \in \DM(\sX)$ and $N \in \DM(\sY)$, there is a natural isomorphism $M \otimes f_!(N) \cong f_!(f^*M \otimes N)$. Dually, for any $M, N \in \DM(\sX)$ one has
        \[
           f^!\iHom(M,N) \cong \iHom(f^*M, f^!N)
        \]
        and for any $M, \, N \in \DM(\sY)$ one has
        \[
            \iHom(f_! M, N) \cong f_*\iHom(M, f^!N).
        \]
        \item \emph{Localization:} Let $i \colon \sZ \to \sX$ be a closed immersion of finite type $\baseS$-schemes and let $j \colon \sU \to \sX$ be the open complement, then for any $M \in \DM(\sX)$ one has distinguished triangles
        \[
            j_!j^*M \to M \to i_!i^* M
        \] 
        and
        \[
            i_*i^!M \to M \to j_*j^! M.
        \]
        \item \emph{Relative purity:} For any smooth morphism $f \colon \sY \to \sX$ there exists a natural isomorphism $f^! \cong f^* \otimes f^!(\ring)$. Moreover, if $f$ is of relative dimension $d$, then $f^!(\ring) = \ring[2d](d)$.
        \item \emph{$\AA^1$-invariance:} For any $\sX \in \Sch_\baseS$, the morphism
        \[
            \pi^* \colon \DM(\sX) \to \DM(\AA^1_{\sX})
        \]
        is fully faithful, where $\pi \colon \AA^1_{\sX} \to \sX$ is the structure map.
        \item \emph{Projective bundle formula:} Let $\sX$ be a smooth scheme and $p \colon \mathrm{P} \to \sX$ be a projective bundle of dimension $n$. Then there exists a canonical isomorphism 
        \[
            p_*(\ring_{\mathrm{P}}) \xrightarrow{\simeq}\bigoplus_{i=0}^n \ring_{\sX}[2i](i).
        \]
    \end{enumerate}
\end{prop}

\begin{defn}
Let $\sX$ be a $\baseS$-finite type scheme. The category $\DTM(\sX, \ring)$ of \emph{Tate motives} is defined as the presentable full subcategory of $\DM(\sX, \ring)$ generated by $\ring(n)$ for $n \in \ZZ$.
\end{defn}

By \cite[Lemma 2.14]{reduced_motives}, there exists a graded algebra $\mathrm{A} \in \grMod_\ring$ such that 
\[
\DTM(\baseS, \ring) \cong \Mod_{\mathrm{A}}(\grMod_\ring)
\]
and, by \cite[Lemma 2.17]{reduced_motives}, this algebra admits an augmentation map $\mathrm{A} \to \ring$. This allows us to define a functor
\[
    \DTM(\baseS) \cong \Mod_\mathrm{A}(\grMod_\ring) \xrightarrow{ \ring \otimes_\mathrm{A} -} \Mod_{\ring}(\grMod_\ring) \cong \grMod_\ring.
\]

\begin{defn}
For $\sX$ a finite type scheme over $\baseS$, we define the category of \emph{reduced motives} as
\[
    \DM_\redm(\sX) \coloneqq \DM(\sX) \otimes_{\DTM(\baseS)} \grMod_\ring.
\]

We similarly define $\DTM_\redm(\sX)$ the category of reduced Tate motives as the presentable full subcategory of $\DM_\redm(\sX)$ generated by the objects $\ring(n)$ for $n \in \ZZ$.
\end{defn}

It follows from the definition that there exists a natural transformation
\[
    \redm \colon \DM(-) \to \DM_\redm(-)
\]
which we call the \emph{reduction functor}.

\begin{rem}
    The $\AA^1$-invariance implies that $\pi^* \colon \DTM_\redm(\sX) \xrightarrow{\simeq} \DTM_\redm(\AA^n_{\sX})$ is an equivalence for any finite type scheme $\sX$, where $\pi \colon \AA^n_\sX \to \sX$ is the structure morphism. 
\end{rem}

From now on, let $\DM_{(\redm)}$ denote either $\DM$ or $\DM_{\redm}$ and $\DTM_{(\redm)}$ denote either $\DTM$ or $\DTM_\redm$.

\begin{rem}\label{rem:weight_structure}
By \cite[Definition and Lemma 4.18]{reduced_motives}, there exists a weight structure on the category $\DTM_{(\redm)}(\baseS, \ring)$ whose heart $\DTM_{(\redm)}(\baseS, \ring)^{\heartsuit_w}$ is the additive retract-closed subcategory generated by $\ring[2n](n)$ for $n \in \ZZ$.
\end{rem}

\begin{example}\label{example:affine_mult}
    As computed in \cite[Lemma 3.3]{reduced_motives}, if $\sX = \AA^r_{\baseS} \times (\GG_{\m, \baseS})^s$, then 
    \[
        \DTM(\sX) = \Mod_{\Sym(\ring (-1) [-1])^{\otimes s}}(\DTM(\baseS))
    \]
    and
    \[
        \DTM_\redm(\sX) = \Mod_{\Sym(\ring \langle -1\rangle [-1])^{\otimes s}}(\grMod_\ring).
    \]
    In particular, for any $i, j \in \ZZ$ and for any $n \neq 0$, one has
    \[
        \Map_{\DM_\redm(\sX)}(\ring[2i](i), \ring[2j+n](j)) = 0.
    \]
\end{example}

From a six functor formalism point of view, $\DM_\redm$ is simply the composition
\[
    \Corr(\Sch_{\baseS}) \xrightarrow{\DM} \Mod_{\DTM(\baseS)}(\Pr^\L_\st) \xrightarrow{ - \otimes_{\DTM(\baseS)} \grMod_\ring} \Mod_{\grMod_\ring}(\Pr^\L_\st).
\]

Since the (relative) tensor product with an idempotent algebra is symmetric monoidal and preserves small colimits, the funtor $\DM_\redm$ still defines a six functor formalism. 

\begin{prop}
    The functor $\DM_\redm \colon \Corr(\Sch_{\baseS}) \to \Mod_{\grMod_\ring}(\Pr^\L_\st)$ satisfies all of the properties listed in Proposition \ref{prop:six_functors_DM}.
\end{prop}

\begin{proof}
    All properties follows from the description of $\DM_\redm$ as a composition of a six functor formalism with a symmetric monoidal continuous functor. This is just an amalgamation of \cite[Proposition 3.7(1), Proposition 3.9, Corollary 3.10 and Corollary 3.11]{reduced_motives}.
\end{proof}

We are going to end this subsection with a comparison result from motives to reduced motives.

\begin{prop}[Proposition 4.20, \cite{reduced_motives}]
    Suppose $\CH^\bullet(\baseS)_\ring = \CH^0(\baseS)_\ring$ (for example, $\baseS$ is the spectrum of a field or $\baseS = \Spec(\ZZ)$). Then the reduction functor gives an isomorphism
    \[
        \Hom_{\DM(\baseS, \ring)}(\ring[2m](m), \ring[2n](n)) \to \Hom_{\DM_{\mathrm{r}}(\baseS, \ring)}(\mathrm{r}(\ring)[2m](m), \mathrm{r}(\ring)[2n](n)).
    \]
    for any $n, \, m \in \ZZ$.
\end{prop}

\begin{proof}
    Since shift and Tate twists are isomorphisms, one may suppose $m = 0$. In this case, we have a usual comparison for the left hand side:
     \[
         \Hom_{\DM(\baseS, \ring)}(\ring, \ring[2n](n)) \cong \CH^n(\baseS, \ring).
    \]
    
    On the other hand, by definition we have
    \[
        \Hom_{\DM_{\mathrm{r}}(\baseS, \ring)}(\ring, \ring[2n](n)) \cong \Hom_{\grMod_\ring}(\ring, \ring[2n](n)).
    \]
    
    Both $\Hom$-sets are trivial if $n \neq 0$ and isomorphic to $\ring$ if $n=0$.
\end{proof}

By definition, $\DTM(\baseS)^{\heartsuit_w}$ is the smallest subcategory containing the objects $\ring[2n](n)$ and stable under direct sum and retractions, therefore a dévissage argument gives the following.

\begin{cor}\label{cor:comparison_base}
For any $M, \, N \in \DTM(\baseS)^{\heartsuit_w}$, one has an isomorphism
\[
        \Hom_{\DM(\baseS, \ring)}(M,N) \to \Hom_{\DM_{\mathrm{r}}(\baseS, \ring)}(M,N).
\]
\end{cor}

\subsection{Motives on stacks}

We are going to describe the extension of the functors $\DM$ and $\DM_\redm$ to algebraic stacks. This was done more generally in \cite{six_functor_formalism}, \cite{non-representable}, \cite{khan-lisse} and \cite{morel-voevodsky_stacks}. We will rewrite some arguments for completeness, and we still work with the same hypothesis for the base scheme $\baseS$. We will denote by $\mathrm{AStk}_\baseS$ the category of algebraic stacks of finite type over $\baseS$.

Once again, denote by $\DM_{(\redm)}$ to the six functor formalism of (reduced) motives. Recall that the main functorial properties are still true in the stacky context.

\begin{prop} \label{prop:six_functors_DM_stacks}
    The 6-functor formalisms $\DM_{(\redm)} \colon \mathrm{Corr}(\Sch_\baseS) \to \Pr^\L_\st$ extend to the category $\mathrm{Corr}(\mathrm{AStk_\baseS})$ and satisfy the following properties:
    \begin{enumerate}[(1)]
        \item For any morphism $f \colon \Y \to \X$, the functor $f^*$ admits a right adjoint $f_*$;
        \item For any morphism $f \colon \Y \to \X$, the functor $f_!$ admits a right adjoint $f^!$;
        \item For any $\X$ algebraic stack of finite type, the category $\DM_{(\redm)}(\X)$ is closed symmetric monoidal;
        \item \emph{Descent:} The functor $\D^* \colon \mathrm{AStk}^{\op}_{\baseS} \to \Pr^\L_\st$ is an NL sheaf\footnote{The definition of the NL topology is given in \cite[Definition 2.1]{non-representable}.}. In particular, given any smooth-Nisnevich atlas $\pi \colon \sX \to \X$, one has
        \[
            \DM_{(\redm)}(\X) = \lim_n \DM_{(\redm)}(\sX^{n/\X}).
        \]
        \item \emph{Base change formulas:} For any Cartesian diagram
        \[
        \begin{tikzcd}
            \Y' \ar["g'"']{d} \ar["f'"]{r} & \X' \ar["g"]{d} \\
            \Y \ar["f"']{r} & \X,
        \end{tikzcd}
        \]
        there exists a natural isomorphism $g_!f'^* \to f^*g'_!$. Dually, there is also a natural isomorphism $g^!f'_* \to f_* g'^!$.
        \item \emph{Projection formulas:} For any $f \colon \Y \to \X$ and for any motives $M \in \DM_{(\redm)}(\X)$ and $N \in \DM_{(\redm)}(\Y)$, there is a natural isomorphism $M \otimes f_!(N) \cong f_!(f^*M \otimes N)$. Dually, for any $M, N \in \DM_{(\redm)}(\X)$ one has
        \[
           f^!\iHom(M,N) \cong \iHom(f^*M, f^!N)
        \]
        and for any $M, \, N \in \DM_{(\redm)}(\Y)$ one has
        \[
            \iHom(f_! M, N) \cong f_*\iHom(M, f^!N).
        \]
        \item \emph{Localization:} Let $i \colon \mathscr{Z} \to \X$ be a closed immersion of finite type algebric stacks and let $j \colon \mathscr{U} \to \X$ be the open complement, then for any $M \in \DM_{(\redm)}(\X)$ one has distinguished triangles
        \[
            j_!j^*M \to M \to i_!i^* M
        \] 
        and
        \[
            i_*i^!M \to M \to j_*j^! M.
        \]
        \item \emph{Relative purity:} For any smooth morphism $f \colon \Y \to \X$ there exists a natural isomorphism $f^! \cong f^* \otimes f^!(\ring)$. Moreover, if $f$ is of relative dimension $d$, then $f^!(\ring) = \ring[2d](d)$.
        \item \emph{Proper pushforward:} For any proper representable morphism $f \colon \Y \to \X$, there exists a natural isomorphism
        \[
            \alpha_f \colon f_! \xrightarrow{\simeq} f_*.
        \]
        \item \emph{$\AA^1$-invariance:} For any algebraic stack of finite type $\X$, the morphism
        \[
            \pi^* \colon \DM_{(\redm)}(\X) \to \DM_{(\redm)}(\AA^1_{\X})
        \]
        is fully faithful, where $\pi \colon \AA^1_\X \to \X$ is the structure morphism.
        \item \emph{Projective bundle formula:} Let $\X$ be a smooth finite type algebraic stack and $p \colon \mathscr{P} \to \X$ be a representable projective bundle of dimension $n$. Then there exists a canonical isomorphism 
        \[
            p_*(\ring_{\mathscr{P}}) \xrightarrow{\simeq}\bigoplus_{i=0}^n \ring_\X[2i](i).
        \]
    \end{enumerate}
\end{prop}

\begin{proof}
    Chowdhury and D'Angelo in \cite[Proposition 3.25]{non-representable} extended $\mathrm{SH}$ to algebraic stacks and their proof works \emph{mutatis mutandis} to $\DM$. In particular, this gives properties (1), (2), (3), (5) and (6). Property (4) is simply \cite[Remark 2.16]{non-representable}. Property (8) is \cite[Corollary 4.25]{non-representable}, since $\DM$ is orientable.

    The other properties may be proven taking an atlas which is an NL-cover. By \cite[Lemma 4.2]{nisnevic_smooth_stacks}, a smooth presentation $\pi \colon \sX \to \X$ is an NL-cover if and only if it is smooth-Nisnevich cover, and by \cite[Theorem A.1]{nisnevic_smooth_stacks} such presentations always exists. We will detail (7) and (10) as examples.

    To prove property (7), let $\sX \to \X$ be a smooth-Nisnevich cover of $\X$, and denote by $\sX_n$ the terms in the \v{C}ech nerve. Similarly, let $j_n \colon \sU_n \hookrightarrow \sX_n \hookleftarrow \sZ_n \colon i_n$ be the pullback of $j \colon \U \hookrightarrow \X \hookleftarrow \Z \colon i$. Let $M \in \DM_{(\redm)}(\X, \ring)$ be a motive, and denote by $M_n$ the pullback to $\sX_n$. Since the localization property is true for schemes, one has distinguished triangles
    \[
        j_{n!}j_n^* M_n \to M_n \to i_{n!}i_n^* M_n.
    \]

    Since we are taking smooth-Nisnevich covers, we have by (4):
    \[
        \DM_{(\redm)}(\X, \ring) \cong \lim \DM_{(\redm)}(\sX_n, \ring),
    \]
    in other words, $M$ is the limit of $M_n$. Taking the limit of the distinguished triangles and using base change, one sees that the sequence
    \[
        j_!j^* M \to M \to i_!i^* M
    \]
    is a distinguished triangle.

    To prove property (10), recall that $\pi^*$ is fully faithful if and only if the unit $\id_{\X} \to \pi_*\pi^*$ is an isomorphism. Let $\sX \to \X$ be a smooth-Nisnevich cover. After pulling back to $\sX$, the unit becomes $\id_\sX \to \pi_{\sX *}\pi_{\sX}^{*}$ which is an isomorphism, since one already knows $\AA^1$-invariance for schemes. Since pullback with respect to a smooth-Nisnevich cover is conservative, the unit $\id_{\X} \to \pi_*\pi^*$ is an isomorphism, concluding our proof.
\end{proof}

As explained in the proof, to compute the category $\DM_{(\redm)}(\X, \ring)$, one first has to choose a smooth-Nisnevich cover $\sX \to \X$. In this case, one has
\[
    \DM_{(\redm)}(\X, \ring) = \lim \DM_{(\redm)}(\sX^{n/\X}, \ring).
\]

\begin{rem}
    One may ask if extending the motives to stacks before or after taking the reduction is the same thing. In other words, if 
    \[
        \left(\lim \DM(\sX^{n/\X}, \ring)\right) \otimes_{\DTM(\baseS, \ring)} \grMod_\ring \cong \lim (\DM(\sX^{n/\X}, \ring) \otimes_{\DTM(\baseS, \ring)} \grMod_\ring).
    \] 
    Since we are taking pullback with respect to smooth morphisms and those admit left adjoint, both limits above are equivalent to the colimits of their left adjoint (where one might use \cite[Proposition 2.1]{reduced_motives} to ensure adjunction after taking the tensor product) and therefore they commute with tensor product.
\end{rem}

Even though they exist for any algebraic stack, smooth-Nisnevich covers are not as abundant as one might hope for. In general, the quotient map $\baseS \to \Bs \G$ is not necessarily a smooth-Nisnevich cover. For this reason, we will mainly restrict ourselves to the case of reductive groups. In this case, we are going to begin studying the torus and the Borel, and then studying the reductive group itself. To get a better understanding of smooth-Nisnevich covers, we will state a lemma which slighly generalizes \cite[Example A.3]{nisnevic_smooth_stacks}.

\begin{lemma}\label{lemma:criterion}
    Let $\G' \subseteq \G$ be an inclusion between two smooth $\baseS$-groups such that the induced map $\H^1((\Spec \mathrm{K})_{\mathrm{et}}, \G') \to \H^1((\Spec \mathrm{K})_{\mathrm{et}}, \G)$ is surjective for every morphism $\Spec \mathrm{K} \to \baseS$ with $\mathrm{K}$ a field, then the map $\Bs \G' \to \Bs \G$ is a representable smooth-Nisnevich cover.
\end{lemma}

\begin{proof}
    A map from $\Spec \mathrm{K}$ to $\Bs \G$ corresponds to a $\G$-torsor $\mathcal{G} \to \Spec \mathrm{K}$. Since the the map on $\H^1$ is surjective, there exists a $\G'$-torsor $\mathcal{G}' \to \mathrm{K}$ such that $\mathcal{G} = \mathcal{G}' \times^{\G'} \G$, but this means that the morphism lifts to a morphism $\Spec \mathrm{K} \to \Bs \G'$.
\end{proof}

\begin{example}\label{ex:GL_n}
    By Hilbert's 90th theorem, the quotient map $\sX \to [\sX/\GL_{n, \baseS}]$ is a smooth-Nisnevich cover. This gives, in particular, an easy way to compute the sheaves on the classifying space of a split torus $\T$.
\end{example}

\begin{example}
    For the same reason, the map $\sX \to [\sX/\mathrm{SL}_{n, \baseS}]$ is also a smooth-Nisnevich cover. 
\end{example}

\begin{example}\label{ex:G_a}
    The quotient map $\sX \to [\sX/\GG_{a, \baseS}]$ is a smooth-Nisnevich cover. The same is also true for any group which is an extension of copies of $\GG_a$, in other words, for any split unipotent group $\sU$.
\end{example}

\begin{prop}
    Let $\Bs$ be a split solvable group, then $\baseS \to \Bs \Bs$ is a smooth-Nisnevich cover.
\end{prop}

\begin{proof}
    For any field $\mathrm{K}$, one has $\H^1((\Spec \mathrm{K})_{\mathrm{et}}, \GG_m) = 0$ and $\H^1((\Spec \mathrm{K})_{\mathrm{et}}, \GG_a) = 0$. Since $\Bs$ admits a filtration whose graded pieces are exactly either $\GG_m$ or $\GG_a$, then $\H^1((\Spec \mathrm{K})_{\mathrm{et}}, \Bs) = 0$ for any field $\mathrm{K}$. This follows, then, from Lemma \ref{lemma:criterion}.
\end{proof}

\begin{prop}\label{prop:normalizer}
    Let $\G$ be a reductive group with maximal torus $\T$. Then the map $\Bs \Nm_\G(\T) \to \Bs \G$ is a smooth-Nisnevich cover.
\end{prop}

\begin{proof}
    In the proof of \cite[Proposition 2.3]{normalizer_torus}, it is shown that the map 
    \[
    \H^1((\Spec \mathrm{K})_{\mathrm{et}}, \Nm_\G(\T)) \to \H^1((\Spec \mathrm{K})_{\mathrm{et}}, \G)
    \]
    is surjective for any field $\mathrm{K}$, and therefore the result follows from Lemma \ref{lemma:criterion}.
\end{proof}

\begin{rem}
    In \cite{normalizer_torus}, it is more generally proven that for any field $\mathrm{K}$, there exists a finite $\mathrm{K}$-group $\Nm$ contained in $\Nm_\G(\T)$ such that $\H^1((\Spec \mathrm{K})_{\mathrm{et}}, \Nm) \to \H^1((\Spec \mathrm{K})_{\mathrm{et}}, \G)$ is surjective. If we want to work over a more general base $\baseS$, one must argue how one may glue such groups to form an $\baseS$-group $\Nm$. If $\baseS$ is a scheme over $\Spec \ZZ[\zeta_\infty]$, then the group $\Nm$ they found is actually constant \cite[3.1. Remark]{normalizer_torus}, and therefore in this case one has a Nisnevich-smooth cover $\Bs \Nm \to \Bs \G$ where $\Nm$ is a constant finite group.
\end{rem}

We conclude this section with a result which is an analogue to homotopy invariance. 

\begin{lemma}\label{lemma:BU-fully-faithful}
    Let $\sU$ be an $\baseS$-group which is isomorphic as a scheme to a vector bundle and let $f \colon \X \to \Y$ be a $\Bs \sU$-torsor between algebraic stacks, then
    \[
    f^* \colon \DM_{(\redm)}(\Y, \ring) \to \DM_{(\redm)}(\X, \ring)
    \]
    is fully faithful.
\end{lemma}

\begin{proof}
    By construction, there exists an NL-cover $\pi \colon \Y' \to \Y$ such that the pullback $f' \colon \X' \to \Y'$ is a trivial $\Bs \sU$-torsor and, thus, there exists a section $\sigma \colon \Y' \to \X'$ which is a $\sU$-torsor. In other words, the map $\sigma$ is a torsor over a vector bundle and therefore $f'^* \colon \DM_{(\redm)}(\Y', \ring) \to \DM_{(\redm)}(\X', \ring)$ is fully faithful. More generally, for any $n \geq 1$, the pullback $f_n \colon \X'^{n/\X} \to \Y'^{n/\Y}$ is a trivial $\Bs \sU$-torsor and there is a similar $\sU$-torsor section $\sigma_n \colon \Y'^{n/\Y} \to \X'^{n/\X}$ which implies that the pullback $f_n^*$ is fully faithful. Since 
    \[
        \DM_{(\redm)}(\X, \ring) = \lim_n \DM_{(\redm)}(\X'^{n/\X}, \ring) \quad \textrm{and} \quad \DM_{(\redm)}(\Y, \ring) = \lim_n \DM_{(\redm)}(\Y'^{n/\Y}, \ring)
    \]
    and since limits commute with mapping spaces, the pullback $f^*$ is fully faithful.
\end{proof}

\subsection{Motives on classifying stacks}

In this subsection, we will describe the category of Tate motives on a classifying space. A similar computation was done by Eberhardt in \cite[Proposition 3.2]{k-motives}. Even though his argument is done over $\ZZ$, it works exclusively for $\K$-motives. This section is a corollary of Krishna's work on higher Chow groups in \cite{krishna_chow}. His computations on Chow groups may be understood as computations on $\Hom$-groups in the category of motives $\DM$. The computations for $\DM_\redm$ are similar.

\begin{defn}\label{defn:torsion}
    Let $\G$ be a split reductive group over $\baseS$ and $\mathrm{B}$ be a Borel subgroup. The torsion number $t_\G \in \ZZ$ is the smallest positive integer such that there exists an element $a \in \Hom_{\DM(\Bs \Bs)}(\ZZ, \ZZ[2n](n))$ for which $p_*(a) = t_\G$ in $\Hom_{\DM(\Bs \G)}(\ZZ, \ZZ) \cong \ZZ$, where $n$ is the codimension of $\Bs$ if $\G$ and $p \colon \Bs \Bs \to \Bs \G$ is the map induced by the inclusion $\Bs \subseteq \G$. 
\end{defn}

\begin{rem}
    Similarly, there exists an element $a' \in \Hom_{\DM_\redm(\Bs \Bs)}(\ZZ, \ZZ[2n](n))$ such that $p_*(a') = t_\G$. Indeed, one may take $a' = \redm(a)$ where $a$ is given by Definition \ref{defn:torsion}.
\end{rem}

Suppose from now on in this subsection that $\ring$ is a $\ZZ[t_\G^{-1}]$-algebra. 

\begin{thm}\label{thm:red_torus_weyl_hom}
    Suppose $\baseS$ is the spectrum of a perfect field and let $\G$ be a split reductive group over $\baseS$ with split maximal torus $\T$. Then the pullback 
    \[
    q^* \colon \Hom_{\DM_{\redm}(\Bs \G, \ring)}(\ring, \ring[2n](n+i)) \to \Hom_{\DM_\redm(\Bs \T, \ring)}(\ring, \ring[2n](n+i))^\sW
    \]
    is an isomorphism for any $n, \, i \in \ZZ$.
\end{thm}

\begin{proof}
    This follows from the same computations as in \cite[Theorem 5.7]{krishna_chow}.
\end{proof}

\begin{cor}\label{prop:class_reduction}
    Suppose $\baseS$ is the spectrum of an algebraically closed field and let $\G$ be a linear algebraic group and let $n, \, m \in \ZZ$. If $\lvert\pi_0(\G)\rvert$ is invertible in $\ring$, then the reduction functor gives an isomorphism
    \[
        \Hom_{\DM(\Bs \G, \ring)}(\ring[2m](m), \ring[2n](n)) \to \Hom_{\DM_{\mathrm{r}}(\Bs \G, \ring)}(\mathrm{r}(\ring)[2m](m), \mathrm{r}(\ring)[2n](n)).
    \]
    for any $n, \, m \in \ZZ$.
\end{cor}

\begin{proof}
    If $\G$ is a split torus, the pullback with respect to the atlas $\pi \colon \baseS \to \Bs \G$ is conservative, and therefore the proposition follows from the fact that a left adjoint is fully faithful if and only if the unit is an isomorphism.

    If $\G$ is a connected split reductive group, then the pullback with respect to the natural map $\Bs \T \to \Bs \G$ induces an isomorphism
    \[
        \Hom_{\DM_{(\redm)}(\Bs \G, \ring)}(\ring[2m](m), \ring[2n](n)) \cong \Hom_{\DM_{(\redm)}(\Bs \T, \ring)}(\ring[2m](m), \ring[2n](n))^\sW.
    \]

    Since the Demazure operators, as defined in \cite[\S 5]{krishna_chow}, are defined using the six operations, they commute with the reduction functor. Since they also generate the Weyl group, this means the result for a connected split reductive group $G$ reduces to the one on the torus.

    If $\G$ is a possibly disconnected reductive group, then the result follows from \cite[Theorem A.2.8]{eq_mot} applied to the natural $\pi_0(\G)$-torsor $\Bs \G^\circ \to \Bs \G$.

    If $\G$ is not necessarily reductive, then $\G_\red \coloneqq \G/\R_\mathrm{u}(\G)$ is reductive and $\R_\mathrm{u}(\G)$ is isomorphic to an affine space. Thus the map $\Bs \G \to \Bs \G_\red$ is a $\Bs \R_\mathrm{u(\G)}$-torsor and the result follows from Lemma \ref{lemma:BU-fully-faithful}.
\end{proof}

\begin{prop}\label{prop:ortho_torus}
    For any $n \in \ZZ$, one has
    \[
        \mathrm{Map}_{\DM_\redm(\Bs \T, \ring)}(\ring, \ring[2n+i](n)) = 0
    \]
    for every $i \neq 0$.
\end{prop}

\begin{proof}
    By Example \ref{ex:GL_n}, the morphism $\baseS \to \Bs \T$ is a Nisnevich-smooth cover of $\Bs \T$, and therefore
    \[
        \DM_\redm(\Bs \T) = \lim_n \DM_\redm(\T^n).
    \]

    By Example \ref{example:affine_mult}, 
    \[
        \mathrm{Map}_{\DM_\redm(\T^n, \ring)}(\ring, \ring[2n+i](n)) = 0,
    \]
    for any $n \in \ZZ$ and $i \neq 0$. Since the mapping space commutes with limits, one has
    \[
        \mathrm{Map}_{\DM_\redm(\Bs \T, \ring)}(\ring, \ring[2n+i](n)) = 0,
    \]
    for any $n \in \ZZ$ and $i \neq 0$, which is what we wanted to prove.
\end{proof}

\begin{prop}\label{prop:orth_borel}
    Suppose $\Bs$ is the Borel subgroup of a connected split reductive group $\G$ and let $n \in \ZZ$. Then
    \[
        \Hom_{\DM_\redm(\Bs \Bs, \ring)}(\ring, \ring[2n+i](n)) = 0
    \]
    for every $i \neq 0$.
\end{prop}

\begin{proof}
    Let $\T$ be the maximal torus of $\G$ and let $\sU$ be the unipotent radical of $\Bs$. By \cite[XXII, Proposition 5.9.5]{SGA3}, there exists a filtrations $\sU = \sU_0 \supseteq \sU_1 \supseteq \dots \supseteq \sU_n = \{1\}$ such that for any $1 \leq j \leq n$ the scheme $\sU_{j-1}/\sU_j$ is a vector bundle and the subgroup $\sU_j$ is normal in $\Bs$. For any $1 \leq j \leq n$, the map $\pi_j \colon \Bs(\T \ltimes \sU_j) \to \Bs(\T \ltimes \sU_{j-1})$ is a $\sU_{j-1}/{\sU_j}$-torsor and by homotopy invariance, the pullback $\pi_j^*$ is fully faithful. Composing all of these maps, one sees the pullback with respect to the natural map $\sigma \colon \Bs \T \to \Bs \Bs$ is fully faithful.
\end{proof}

\begin{prop}\label{prop:levi_decomp}
Let $\G$ and $\L$ be algebraic groups over $\baseS$ such that $\L = \G / \sU$, for $\sU$ is a split unipotent group over $\baseS$. Then for every $n, m \in \ZZ$, the pullback with respect to the morphism $\sigma \colon \Bs \G \to \Bs \L$ is fully faithful.
\end{prop}

\begin{proof}
This follows Lemma \ref{lemma:BU-fully-faithful}.
\end{proof}

\begin{prop}\label{prop:orth_class}
    Let $\baseS$ be the spectrum of a perfect field and $\G$ be a connected split reductive group. For any $n \in \ZZ$, one has
    \[
        \Hom_{\DM_\redm(\Bs \G, \ring)}(\ring, \ring[2n+i](n)) = 0
    \]
    for every $i \neq 0$.
\end{prop}

\begin{proof}
    If $\G$ is a split torus, this is just Proposition \ref{prop:ortho_torus}. For a general connected reductive group, this follows from Theorem \ref{thm:red_torus_weyl_hom}.
\end{proof}

\begin{prop}\label{prop:orth_class_disc}
    Let $\baseS$ be the spectrum of a perfect field and let $\G$ be a (possibly disconnected) reductive group such that $\left| \pi_0(\G)\right|$ is invertible in $\ring$. For any $n \in \ZZ$, one has
    \[
        \Hom_{\DM_\redm(\X, \ring)}(\ring, \ring[2n+i](n)) = 0
    \]
    for every $i \neq 0$.
\end{prop}

\begin{proof}
    The component group $\pi_0(\G)$ is a finite group and the map $\pi \colon \Bs \G^\circ \to \Bs \G$ is a $\pi_0(\G)$-torsor. In this case, the composition
    \[
        \ring_{\Bs \G} \to \pi_*\pi^* \ring_{\Bs \G} \to \ring_{\Bs \G}
    \]
    is multiplication by $\left| \pi_0(\G)\right|$ and, therefore, is invertible by the hypothesis on $\ring$. This means the unit $\ring_{\Bs \G} \to \pi_*\pi^* \ring_{\Bs \G}$ is a split monomorphism, and therefore, once again, $\pi^*$ is faithful, concluding our proof.
\end{proof}

\begin{cor}\label{cor:orth_linear_gp}
    Let $\baseS$ be the spectrum of an algebraically closed field. Then, for any linear algebraic group $\G$ such that $\left| \pi_0(\G)\right|$ is invertible in $\ring$, one has
    \[
        \Hom_{\DM_\redm(\Bs \G, \ring)}(\ring, \ring[2n+i](n)) = 0
    \]
    for every $i \neq 0$.
\end{cor}

\begin{proof}
    In this case, the quotient $\G_{\textrm{red}} \coloneqq  \G/\R_\mathrm{u}(\G)$ is split reductive and $\R_\mathrm{u}(\G)$ is a split unipotent. The result follows from Lemma \ref{lemma:BU-fully-faithful}, Proposition \ref{prop:levi_decomp} and Proposition \ref{prop:orth_class_disc}.
\end{proof}

\subsection{Motives on quotient stacks}

In this section, $\baseS$ will denote the spectrum of an algebraically closed field to ensure every reductive group is split.

\begin{defn}
    Let $\X$ be an algebraic stack of finite type over $\baseS$. We say $M \in \DM_{(\redm)}(\X)$ is $*$-pointwise pure Tate (resp. $!$-pointwise pure Tate) if for every $S$-point $x \colon \baseS \to X$, one has $x^* M \in \DTM_{(\redm)}(\baseS)^{\heartsuit_w}$ (resp.  $x^! M \in \DTM_{(\redm)}(\baseS)^{\heartsuit_w}$). We say $M \in \DM_{(\redm)}(\X)$ is pointwise pure Tate if it is both $*$-pointwise pure Tate and $!$-pointwise pure Tate.
\end{defn}

\begin{lemma}\label{lemma:reduction_iso}
    Let $\G$ be a linear algebraic group acting on a finite type scheme $\sX$ with finitely many orbits and let $M, \, N \in \DM_{(\redm)}(\sX/\G, \ring)$ be $*$- and $!$-pointwise pure Tate motives, respectively. If for every $x \in \sX$ the integers $t_{\Ct_\G(x)_\red}$ and $\lvert \pi_0 \Ct_\G(x) \rvert$ are invertible in $\ring$, then the reduction functor gives an isomorphism
    \[
        \Hom_{\DM(\sX/\G, \ring)}(M, N) \to \Hom_{\DM_{\mathrm{r}}(\sX/\G, \ring)}(\mathrm{r}(M), \mathrm{r}(N)).
    \]
\end{lemma}

\begin{proof}
    We do an induction on the number of orbits. If there is just one, then this is simply Proposition \ref{prop:class_reduction}. If there are more orbits, choose one closed orbit $\sZ$, which exists by the closed orbit lemma, and consider the open complement $\sU$. If $i \colon \sZ/\G \to \sX/\G$ and $j \colon \sU/\G \to \sX/\G$ are the corresponding closed and open immersions between quotients stacks, then one has a distinguished triangle
    \[
        j_! j^* M \to M \to i_! i^* M,
    \] 
    which once applied $\Hom_{\DM_{\mathrm{r}}(\sX/\G, \ring)}(-, N[i])$ gives an exact sequence in which one may apply the induction hypothesis and the five lemma to conclude the result.
\end{proof}

\begin{lemma}\label{lemma:reduced_tilting}
    Let $\G$ be a linear algebraic $\baseS$ group acting on a finite type $\baseS$-scheme $\sX$ with finitely many orbits and let $M, \, N \in \DM_{(r)}(\sX/\G, \ring)$ be $*$- and $!$-pointwise pure Tate motives, respectively. If for every $x \in \sX$ the integers $t_{\Ct_\G(x)_\red}$ and $\lvert \pi_0 \Ct_\G(x) \rvert$ are invertible in $\ring$, then
    \[
        \Hom_{\DM_{\mathrm{r}}(\sX/\G, \ring)}(M,N[i]) = 0
    \]
    for every $i \neq 0$.
\end{lemma}

\begin{proof}
    We are going to do an induction on the number of orbits. If there's just one orbit, then $\sX/\G = \Bs \Ct_\G(x)$ and we are reduced to Corollary \ref{cor:orth_linear_gp}. If there are more orbits, choose one close orbit $\sZ$, which exists by the closed orbit lemma, and consider the open complement $\sU$. If $i \colon \sZ/\G \to \sX/\G$ and $j \colon \sU/\G \to \sX/\G$ are the corresponding closed and open immersions between quotients stacks, then one has a distinguished triangle
    \[
        j_! j^* M \to M \to i_! i^* M,
    \] 
    which once applied $\Hom_{\DM_{\mathrm{r}}(\sX/\G, \ring)}(-, N[i])$ gives an exact sequence
    \[
        \Hom_{\DM_{\mathrm{r}}(\sU/\G, \ring)}(j^*M, j^! N[i]) \to \Hom_{\DM_{\mathrm{r}}(\sX/\G, \ring)}(M,N[i]) \to \Hom_{\DM_{\mathrm{r}}(\sZ/\G, \ring)}(i^*M, i^! N[i]).
    \]

    By the induction hypothesis, the left term vanishes, and again by Corollary \ref{cor:orth_linear_gp}, the second term vanishes too, which completes the proof.
\end{proof}

\section{The integral Springer correspondence}

In this section, we are going to suppose the base scheme $S$ is the spectrum of an algebraically closed field. In particular, all of the reductive groups will be split and all connected smooth unipotent groups will be isomorphic to affine spaces.

\subsection{The Springer setup}

This subsection will follow \cite[\S 4]{motivic_springer}, but we are going to prove the results with integral coefficients. We work both with reduced and unreduced motives, but the final result will be for reduced motives only.

Let $\G$ be a linear algebraic $\baseS$-group, and let $\mu_i \colon \widetilde{\N}_i \to \N$ be a collection of $\G$-equivariant proper maps between quasi-projective varieties where each $\widetilde{\N}_i$ is smooth and connected for every $i \in I$. We will consider the full subcategory $\Tilt^{\Spr}_{(\redm)}$ (resp. $\widehat{\Tilt}^{\Spr}_{(\redm)}$) of $\DM_{(\redm)}(\N/\G, \ring)$ whose objects are $\mu_{i!} \ring_{\widetilde{\N}_i/\G}$ for every $i \in I$ (resp. $\mu_{i!} \ring [2n](n)$ for every $i \in I$ and $n \in \ZZ$). 

\begin{defn}
    The category of (reduced) \emph{Springer motives} $\DM^{\Spr}_{(r)}(\N/\G, \ring)$ is the thick full subcategory of $\DM_{(r)}(\N/\G)$ spanned by $\widehat{\Tilt}_{(\redm)}^\Spr$. 
\end{defn}

We are going to use the following three hypotheses:
\begin{enumerate}
    \item[(PT)] For each $x \in \N$, the motive\footnote{If $a \colon \sX \to \baseS$ is an $\baseS$-scheme, we define its motive by  $\M(\sX) \coloneqq a_* \ring$.} $\M(\mu_i^{-1}(x))$ is pure Tate;
    \item[(FO)] For each $i$, the closed subvarieties $\mu_i(\widetilde{\N}_i) \subset \N$ have a finitely many orbits;
    \item[(RG)] For each $x \in \N$, the cardinality $\left|\pi_0(\Ct_\G(x))\right|$ and the torsion number $t_{\Ct_\G(x)_\red}$ are invertible in $\ring$.
\end{enumerate}

\begin{rem}
    The hypothesis of having a finite number of orbits is not optimal, but is enough for our purposes and present in many cases of interest.
\end{rem}

As the notation suggests, $\widehat{\Tilt}^\Spr_\redm$ will form a tilting\footnote{The collection $\widehat{\Tilt}^\Spr$ is negative, but it is not tilting. This is the case because working with reduced motives kills the higher Chow groups, which are the pieces of the higher Homs.} collection and, therefore, will define a weight structure on $\DM^{\Spr}_\redm(\N/\G, \ring)$. 

\begin{lemma}\label{lemma:pointwise_springer}
    Under hypotheses (PT), (FO) and (RG), the objects of $\widehat{\Tilt}^{\Spr}_{(\redm)} \subset \DM_{(\redm)}(\N/\G)$ are pointwise pure Tate. In particular, the full subcategory $\widehat{\Tilt}^{\Spr}_\redm \subset \DM_{\redm}(\N/\G)$ is tilting.
\end{lemma}

\begin{proof}
    We need to prove the motives $\mu_{i!}\ring[2n](n)$ are all pointwise pure Tate. By definition, this means that for any $i_x \colon x \to \N$, one has $i_x^* \mu_{i!}\ring[2n](n)$ and $i_x^!\mu_{i!}\ring[2n](n)$ are pure Tate, and this follows from base change on the diagram
    \[
    \begin{tikzcd}
        \mu_i^{-1}(x) \ar[hook]{r} \ar["\mu_i"']{d} & \widetilde{\N}_i \ar["\mu_i"]{d} \\
        \{x\} \ar["i_x"', hook]{r} & \N
    \end{tikzcd}
    \]
    and Verdier duality.

    Finally, let $M_1 = \mu_{i!}\ring[2n_1](n_1)$ and $M_2 = \mu_{j!}\ring[2n_2](n_2)$. By properness of $\mu_i$ and $\mu_j$, the morphism $m_{i,j} \colon \N_{i,j} = \mu_i(\widetilde{\N}_i) \cup \mu_j(\widetilde{\N}_j) \hookrightarrow \N$ is a closed immersion and $\N_{i,j}$ has finitely many orbits by the hypothesis (FO). Since $M$ and $N$ are supported on $\N_{i,j}$ one has
    \[
        \Hom_{\DM_{(\redm)}(\N/\G)}(M,N[n]) = \Hom_{\DM_{(\redm)}(\N_{i,j}/\G)}(m_{i,j}^* M, m_{i,j}^! N).
    \]

    Since $m_{i,j}^* M$ is $*$-pure Tate and $m_{i,j}^! N$ is $!$-pure Tate, the final assertion follows from Lemma \ref{lemma:reduced_tilting}.
\end{proof}

\begin{cor}\label{cor:springer}
    Under hypotheses (PT), (FO) and (RG), there is an equivalence of categories
    \[
        \DM^{\Spr}_\redm(\N/\G, \ring) \cong \mathrm{grPerf}(\End^\bullet_{\DM_\redm(\N/\G, \ring)}(\Tilt^\Spr_\redm)).
    \]
\end{cor}

\begin{proof}
    This follows from Lemma \ref{lemma:pointwise_springer} and Corollary \ref{thm:tilting_algebra_2}.
\end{proof}

Next, we shall describe the algebra $\End^\bullet_{\DM_\redm(\N/\G, \ring)}(\Tilt^\Spr_\redm)$. The first step is comparing with the non-reduced motives.

\begin{lemma}\label{lemma:iso_end}
    Under hypotheses (PT) and (FO) one has an isomorphism of algebras
    \[
        \End^\bullet_{\DM(\N/\G, \ring)}(\Tilt^\Spr) = \End^\bullet_{\DM_\redm(\N/\G, \ring)}(\Tilt^\Spr_\redm).
    \]
\end{lemma}

\begin{proof}
    Denote by $\mathrm{E}^\bullet$ and $\mathrm{E}^\bullet_\redm$ the two algebras in consideration. Since the objects of $\widehat{\Tilt}^{\Spr}$ are pointwise pure Tate, by Lemma \ref{lemma:reduction_iso}, one has for any $i, \, j \in I$ an isomorphism
\[
    \Hom_{\DM(\N/\G, \ring)}(\mu_{i!} \ring, \mu_{j!} \ring [2n](n)) = \Hom_{\DM_{\redm}(\N/\G, \ring)}(\redm(\mu_{i!} \ring), \redm(\mu_{j!} \ring [2n](n))),
\]
which implies $\mathrm{E}^n = \mathrm{E}^n_\redm$. Moreover, since $\redm$ is a functor, it is compatible with composition, giving the isomorphism of algebras.
\end{proof}

\begin{defn}
    The \emph{(reduced) motivic extension algebra} is the algebra
    \[
        \mathrm{E}_{(\redm)}^\bullet = \End^\bullet_{\DM_{(\redm)}(\N/\G, \ring)}(\Tilt^\Spr_{(\redm)}).
    \]
    This is a graded algebra whose $n$-th graded part is given by
    \[
        \mathrm{E}^n_{(\redm)} = \bigoplus_{i, j} \Hom_{\DM_{(\redm)}(\N/\G, \ring)}(\mu_{i!} \ring, \mu_{j!} \ring [2n](n)).
    \]
\end{defn}

With this notation in mind, Corollary \ref{cor:springer} and Lemma \ref{lemma:iso_end} imply
\[
    \DM^{\Spr}_\redm(\N/\G, \ring) \cong \mathrm{grPerf}(\mathrm{E}^\bullet).
\]

As usual, the algebra $\mathrm{E}^\bullet$ admits a geometric description using Chow groups. In the context of quotient stacks over a field, this was done, for example, by Aranha and Chowdhury in \cite{Chow-weight-stacks}. Let $\Z_{i, j} = \widetilde{\N}_i \times_\N \widetilde{\N}_j$ and $d_i = \dim \widetilde{\N}_i$.

\begin{prop}[\cite{Chow-weight-stacks}, Theorem 5.2]
    For any $n \in \ZZ$, there is a natural isomorphism
    \[
        \Hom_{\DM(\N/\G, \ring)}(\mu_{i!}[2m](m) \ring, \mu_{j!}\ring[2n](n)) = \CH_{d_j-n+m}^G(\Z_{i,j})_\ring.
    \]
\end{prop}

As stated, this proposition just gives an isomorphism of $\Lambda$-modules
\[
    \mathrm{E}^n = \bigoplus_{i, j} \CH_{d_j-n}^\G(\Z_{i,j})_\ring.
\]
The same proof as in \cite[Proposition 2.39 and Proposition 3.16]{jin_borel-moore} (see also \cite[Theorem 3.17]{jin_borel-moore}) shows this equivalence is in fact compatible with multiplication, giving the following result.

\begin{cor}\label{cor:iso_alg}
    There is an isomorphism of graded algebras
    \[
        \mathrm{E}^\bullet \cong \bigoplus_{i,j} \CH^\G_{d_j - \bullet} (\Z_{i,j})_\ring.
    \]
\end{cor}

Putting everything together, we have our main theorem.

\begin{thm}\label{thm:springer_setup}
    Under the hypothesis (PT), (FO) and (RG), there is an equivalence of categories
    \[
        \DM^{\Spr}_\redm(\N/\G, \ring) \cong \mathrm{grPerf}\left(\bigoplus_{i,j} \CH^\G_{d_j - \bullet}(\Z_{i,j})_\ring\right).
    \]
\end{thm}

\begin{proof}
    This follows from Corollary \ref{cor:springer}, Lemma \ref{lemma:iso_end} and Corollary \ref{cor:iso_alg}.
\end{proof}

\begin{rem}
    During the writing of this paper, the author became aware that (a stronger version of) the same result was claimed by Eberhardt in \cite[\S 3.4]{k-motives}. Corollary \ref{prop:class_reduction} and Corollary \ref{cor:orth_linear_gp} are, together, analogues of \cite[Proposition 3.2]{k-motives} and at the time of writing it is not clear to the author how the proof of Proposition 3.2 in \textit{loc. cit.}, which uses explicit computations for $\K$-theory, generalizes to $\DM$. With both corollaries (or, equivalently, with Proposition 3.2 in \textit{loc. cit.}) as input, our proof, inspired by \cite{motivic_springer}, is essentially the same as the one in \cite{k-motives}.  
\end{rem}

If $\baseS = \Spec \overline{\FF}_p$ and $\ring = \QQ$, this recovers the main theorem of \cite{motivic_springer}. As a final comment, in the case of an algebraically closed field of characteristic zero, the weight structure determined by $\widehat{\Tilt}^{\Spr}_{(\redm)}$ is just the restriction of the Chow weight structure defined by Aranha and Chowdhury in \cite{Chow-weight-stacks}.

\begin{prop}
    If $\baseS$ is the spectrum of an algebraically closed field of characteristic zero, then the motives $\mu_{i!} \ring(n)[2n]$ are in the heart of the Chow weight structure for any $n \in \ZZ$.
\end{prop}

\begin{proof}
    Since $\widetilde{\N}_i$ and $\N$ are quasi-projective varieties over $\baseS$, the maps $\mu_i \colon \widetilde{\N}_i/\G \to \N/\G$ are projective. By definition of the Chow weight structure, one has $\mu_{i!} \ring(n)[2n] \in \DM_{\gm}(\N/\G)^{\heartsuit_w}$ for any $n \in \ZZ$.
\end{proof}

\subsection{The Springer resolution}

Let $\G$ be a reductive group over an algebraically closed field $\basef$ of characteristic exponent $p$, let $\N_\G$ be its nilpotent cone and let $\ring$ be a $\ZZ[\frac{1}{p}]$-algebra. We are going to consider the case of the Springer resolution $\mu \colon \widetilde{\N}_\G \to \N_\G$ and the pullback of this map with itself is the Steinberg variety $\mathrm{St}_\G \coloneqq \widetilde{\N}_\G \times_{\N_\G} \widetilde{\N}_\G$. Since we are going to use \cite[Theorem 1.1]{springer_motives}, in this section we work with the hypothesis that either $p=1$ or the following holds:
\begin{enumerate}
    \item The prime $p$ is a good prime for every classical constituent of $\G$;
    \item One has\footnote{Eberhardt works with a slighly stronger condition, but as was remarked in the paper, one only needs the Jacobson-Morozov theorem, which is true for $p>h$, as proved in \cite{jacobson-morozov}.} $p>h$ where $h$ is the maximum of all Coxeter numbers of exceptional constituent of $\G$.
\end{enumerate}

Before stating the result, we need an analogue of \cite[Definition 6.1]{modular_springer_iii} for an arbitrary ring, instead of a field.

\begin{defn}
    Let $\G$ be a (possibly disconnected) reductive group over $\basef$ of characteristic exponent $p$. We say the $\ZZ[\frac{1}{p}]$-algebra $\ring$ is \emph{rather good} (resp. \emph{easy}) for $\G$ if $\left| \mathrm{A}_\G(e) \right|$ is invertible in $\ring$ for every $e \in \N_\G$ (resp. if $\left|\sW_\G\right|$ is invertible in $\ring$).
\end{defn}

Before stating the main result, we need the following lemma to make sense of the parabolic induction.

\begin{lemma}
    If $\ring$ is rather good for $\G$, then it is rather good for every Levi subgroup $\L$ of $\G$.
\end{lemma}

\begin{proof}
    Suppose $\ring$ is rather good for $\G$ and let $\L$ be a Levi subgroup of $\G$. If follows from \cite[Lemma 3.10]{modular_springer_ii} that for any element $x \in \N_\L$ there exists $y \in \N_\G$ such that $\left| \mathrm{A}_\L(x)\right|$ divides $\lvert\mathrm{A}_\G(y)\rvert$, indeed one may take $y$ in the Lusztig-Spaltenstein induction of the $\L$-orbit of $x$. In particular $\left| \mathrm{A}_\L(x)\right|$ is invertible in $\ring$, concluding the proof.
\end{proof}

\begin{thm}\label{thm:springer_correspondence}
    Suppose $\ring$ is rather good for $\G$. 
    Then there exists a canonical equivalence of categories
    \[
        \DM^{\Spr}_{\redm}(\N_\G/\G, \ring) \cong \mathrm{grPerf}(\CH^\bullet(\St_\G/\G)_\ring)
    \]
    which is compatible with parabolic induction.
\end{thm}

\begin{proof}
    The pure Tate property of the Springer fibers was proved by Eberhardt in \cite{springer_motives} and the fact that the nilpotent cone has finitely many orbits is a standard result, therefore the result follows from Theorem \ref{thm:springer_setup}. To understand the compatibility with the parabolic induction, let $\P$ be a parabolic subgroup of $\G$ with Levi factor $\L$. The diagram
    \[
        \N_\L/\L \xleftarrow{\pi} \N_\P/\P \xrightarrow{\mu} \N_\G/\G,
    \]
    in which $\pi$ is smooth of relative dimension 0 and $\mu$ is proper, defines a weight exact functor
    \[
        \ind_{L \subset P}^G \coloneqq \pi^* \mu_! \colon \DM^\Spr_\redm(\N_\L/\L, \ring) \to \DM^\Spr_\redm(\N_\G/\G, \ring).
    \] 
    
    Similarly, one has a similar diagram taking the pullback
    \[
        \St_\L/\L \xleftarrow{\pi_\St} \St_\P/\P \xrightarrow{\mu_\St} \St_\G/\G,
    \]
    where $\pi_\St$ is smooth and $\mu_\St$ is proper, which induces a morphism of algebras
    \[
        \ind_{L \subset P, \St}^G \coloneqq \pi_\St^* \mu_{\St !} \colon \CH^\bullet(\St_\G/\G)_\ring \to \CH^\bullet(\St_\G/\G)_\ring.
    \]
    
    Now, remark the equivalence of categories is given by a composition of equivalences
    \[
        \DM^{\Spr}_{\redm}(\N_\G/\G, \ring) \xrightarrow{\simeq} \Ch^b(\DM^{\Spr}_{\redm}(\N_\G/\G, \ring)^{\heartsuit_w}) \xrightarrow{\simeq} \mathrm{grPerf}(\CH^\bullet(\St_\G/\G)_\ring).
    \]

    The functor on the left is an commutes with parabolic induction because of Lemma \ref{thm:tilting_equivalence_2}, and the right one commutes because of comparison results between Borel-Moore homology and Chow groups, as done in \cite[Proposition 2.3.12 and Proposition 2.3.13]{jin_borel-moore}.
\end{proof}

The following table gives a list of bad primes \cite[I.4.3]{conjugacy_classes}, torsion primes (i.e. prime factors of the torsion number) \cite[Corollary 1.13]{torsion_steinberg}, the order of the Weyl group \cite[\S 2.11]{humphreys_coxeter} and the Coxeter number $h$ \cite[\S 3.18]{humphreys_coxeter} for irreducible root systems and the prime factors of $\mathrm{A}_\G(\Or)$ for $\Or$ ranging through the nilpotent orbits of $\G$ for $\G$ simple group of adjoint type in good characteristic \cite[Chapter 13]{carter_finite}.

\begin{center}
\begin{tabular}{||c | c | c | c | c| c||} 
 \hline
 type & bad & torsion & $\lvert \sW \rvert$ & $h$ & $\mathrm{A}_\mathrm{G}(\Or)$ \\ [0.5ex] 
 \hline\hline
 $\mathrm{A}_n$ & - & - & $(n+1)!$ & $n+1$ & 1\\ 
 \hline
 $\mathrm{B}_n$ & 2 & 2 & $2^n \cdot n!$ & $2n$ & 2 \\
 \hline
 $\mathrm{C}_n$ & 2 & - &  $2^n \cdot n!$ & $2n$ & 2 \\
 \hline
 $\mathrm{D}_n$ & 2 & 2 & $2^{n-1} \cdot n!$ & $2(n-1)$ & 2 \\
 \hline
 $\mathrm{E}_6$ & 2, 3 & 2, 3 & $2^7 \cdot 3^4 \cdot 5$ & 12 & 2, 3  \\
 \hline
 $\mathrm{E}_7$ & 2, 3 & 2, 3 & $2^{10} \cdot 3^4 \cdot 5 \cdot 7$ & 18 & 2, 3  \\
 \hline
 $\mathrm{E}_8$ & 2, 3, 5 & 2, 3, 5 & $2^{14} \cdot 3^5 \cdot 5^2 \cdot 7$ & 30 & 2, 3, 5 \\
 \hline
 $\mathrm{F}_4$ & 2, 3 & 2, 3 & $2^7 \cdot 3^2$ & 12 & 2, 3 \\
 \hline
 $\mathrm{G}_2$ & 2, 3 & 2 & $2^2 \cdot 3$ & 6 & 2, 3 \\ 
 \hline
\end{tabular}
\end{center}

In particular, the correspondence we have found works with coefficients $\ring = \ZZ$ for $\G$ of type $\mathrm{A}_n$, with coefficients $\ring = \ZZ[\tfrac{1}{2}]$ for $\G$ classical group and works with coefficients $\ring = \ZZ[\frac{1}{30}]$ for any connected reductive group $\G$. 

\subsection{Quiver Hecke algebras} 

The Springer setup may also be applied to relate quiver varieties and quiver Hecke algebras. We give a sketch of how to adapt it to this context. Let $\basef$ be an algebraically closed field of exponent characteristic $p$ and $\mathrm Q$ be a quiver with finite sets of vertices $\mathrm{Q}_0$ and arrows $\mathrm{Q}_1$ with source and target maps
\[
    s,t \colon \mathrm{Q}_1 \to \mathrm{Q}_0.
\]

Let $\Gamma \coloneqq \ZZ_{\geq 0}[\mathrm{Q}_0]$ (resp. $\Gamma^+ \coloneqq \ZZ_{\geq 0}[\mathrm{Q}_0]\backslash\{0\}$) be the set of (non-trivial) dimension vectors. The \emph{dimension vector} of a $\mathrm{Q}_0$-graded $\basef$-vector space $\mathrm{V}$ is the tuple $\dim \mathrm{V} \coloneqq (\dim(\mathrm{V}_i))_{i \in \mathrm{Q}_0} \in \Gamma$. For such $\mathrm V$, we define by
\[
    \Rep(\mathrm V ) \coloneqq \prod_{\varphi \in \mathrm{Q}_1} \Hom(\mathrm{V}_{s(\varphi)}, \mathrm{V}_{t(\varphi)})
\]
the space of $\mathrm Q$-representations with underlying $\basef$-vector space $\mathrm V$. If $\mathrm V$ is the standard $\mathrm{Q}_0$-graded $\basef$-vector space with dimension vector $\dv$, we denote $\Rep(\dv) = \Rep(\mathrm V )$. The group
\[
    \GL(\dv) = \GL(\mathrm V ) = \prod_{i \in \mathrm{Q}_0} \GL(\mathrm{V}_i)
\]
acts on $\Rep(\dv)$ in a natural way.

For $\ell \in \ZZ_{\geq 0}$, a \emph{composition} of length $\ell$ of a dimension vector $\dv$ is a tuple $\pdv = (\pdv^j) \in (\Gamma^+)^\ell$ whose sum is $\dv$. A composition $\pdv$ is \emph{complete} if every $\pdv^j$ is a unit vector. We denote by $\Comp(\dv)$ (resp. $\Compf(\dv)$) the set of (complete) compositions of $\dv$. 

Continuing with the same notation, a partial flag $\underline{\mathrm V}$ of $\mathrm V$ of type $\pdv$ is a sequence of $\mathrm{Q}_0$-graded vector spaces
\[
    0 = \mathrm{V}^0 \subset \mathrm{V}^1 \subset \dotsc \subset \mathrm{V}^{\ell_{\pdv}} = \mathrm V 
\]
such that $\dim(\mathrm{V}^j/ \mathrm{V}^{j-1}) = \pdv^j$. There exists a smooth projective variety parametrizing such partial flags, which we will denote by $\Fl(\pdv)$.

If $\rho \in \Rep(\dv)$ is a $\mathrm Q$-representation, a partial flag $\underline{\mathrm V}$ is called \emph{strictly $\rho$-stable} if
\begin{equation} \label{eq:strictly_stable}
    \rho(\mathrm{V}^j) \subseteq \mathrm{V}^{j-1} \quad \textrm{ for every } 1 \leq j \leq \ell_{\pdv}.
\end{equation}

We consider the smooth $\basef$-variety
\[
    \mathscr{Q}(\pdv) \coloneqq \{(\rho, \underline{\mathrm V}) \in \Rep(\dv) \times \Fl(\pdv) \mid \underline{\mathrm V} \textrm{ is strictly $\rho$-stable}\}.
\]
It admits a natural $\GL(\dv)$-action and the projection 
\[
    \mu_{\pdv} \colon \mathscr{Q}(\pdv) \to \Rep(\dv)
\]
defines a $\GL(\dv)$-equivariant proper map. The fiber of a representation $\rho \in \Rep(\dv)$ is called a (strict) partial quiver flag variety and denoted by
\[
    \Fl_{\mathrm{str}}(\rho, \pdv) \coloneqq \{ \underline{\mathrm V} \in \Fl(\pdv) \mid \underline{\mathrm V} \textrm{ is strictly $\rho$-stable}\}.
\]

Before stating the theorem, we just need two more notations. First of all, we will denote by $\mathrm{R}_\dv$ the KLR algebra, originally defined by Khovanov and Lauda in \cite[\S 2.1]{KL} and Rouquier in \cite[\S 3.2]{rouquier}. We will say a $\ZZ[\frac{1}{p}]$-algebra $\ring$ is \emph{rather good} for a quiver $\mathrm Q$ if for every nilpotent representation $\rho$, the cardinality $\left|\pi_0(\Ct_{\GL(\dv)}(\rho))\right|$ and the torsion number $t_{\Ct_{\GL(\dv)}(\rho)}$ are invertible in $\ring$.

\begin{thm}
    Suppose $\mathrm Q$ is of Dynkin type or of type $\widetilde{\mathrm A}$ with cyclic orientation and suppose $\ring$ is rather good for $\mathrm Q$. Then there exists an equivalence of categories
    \[
        \DM^{\Spr}_\redm(\Rep(\dv)/\GL(\dv), \ring) \cong \mathrm{grPerf}(\mathrm{R}_\dv),
    \]
    where the Springer motives are taken with respect to the maps $\mu_{\pdv} \colon \mathscr{Q}(\pdv) \to \Rep(\dv)$ when $\pdv$ ranges through $\pdv \in \Compf(\dv)$.
\end{thm}

\begin{proof}
    Since, by \cite[Theorem 5.1]{Chow-weight-stacks}, equivariant Chow groups and equivariant Borel-Moore homology coincide, it follows from the same computations as in \cite[Theorem 3.6]{klr-algebras_computation} that the extension algebra is isomorphic to $\mathrm{R}_\dv$.

    By Corollary \ref{cor:springer}, to prove the equivalence, it suffices to prove condition (PT) and (FO). When $Q$ is as in the statement, Zhou proved in \cite{affine_paving_quivers} that the (strict) partial flag varieties admit affine pavings and are, therefore, pure Tate.
    
    If $\mathrm Q$ is of Dynkin type, there are finitely many isomorphism classes of $\mathrm Q$-representations with a fixed dimension vector, and therefore $\Rep(\dv)$ has only finitely many $\GL(\dv)$-orbits. If $\mathrm Q$ is of type $\widetilde{\mathrm A}$ with cyclic orientation, it follows from (\ref{eq:strictly_stable}) that any representation in the image of $\mu_{\pdv}$ is nilpotent and it follows from \cite[Theorem 3.24]{lectures_schiffman} that there are just finitely many of them with a fixed dimension vector.
\end{proof}

Instead of considering just the complete compositions $\pdv \in \Compf(\dv)$, one may consider all compositions $\pdv \in \Comp(\dv)$. In this case, the extension algebra is isomorphic to (an integral version of) the \emph{Schur algebra} $\mathrm{A}_\dv$ defined by Stroppel and Webster in \cite{stroppel-webster} and we get the following result.

\begin{thm}
    Suppose $\mathrm Q$ is of Dynkin type or of type $\widetilde{\mathrm A}$ with cyclic orientation and suppose $\ring$ is rather good for $\mathrm Q$. Then there exists an equivalence of categories
    \[
        \DM^{\Spr}_\redm(\Rep(\dv)/\GL(\dv), \ring) \cong \mathrm{grPerf}(\mathrm{A}_\dv),
    \]
    where the Springer motives are taken with respect to the maps $\mu_{\pdv} \colon \mathscr{Q}(\pdv) \to \Rep(\dv)$ when $\pdv$ ranges through $\pdv \in \Comp(\dv)$.
\end{thm}

\begin{proof}
    In the proof of the previous theorem, we did not use the fact $\pdv$ was complete. In particular, the same proof applies here.
\end{proof}

\appendix

\section{Weight structures}

We are going to recall some fundamental results on the theory of weight structures, which allows us to both define and use weight structures.

\begin{defn}
    A weight structure on a stable $\infty$-category $\C$ is a pair of retract-closed full subcategories $(\C_{w \geq 0}, \C_{w\leq 0})$ such that, if we call $\C_{w \geq n} = \Sigma^n\C_{w \geq 0}$ and $\C_{w\leq n} = \Sigma^{n} \C_{w \leq n}$ for $n \in \ZZ$, then:
    \begin{enumerate}[(a)]
        \item $\C_{w \geq 1} \subset \C_{w\geq 0}$ and $\C_{w \leq -1} \subset \C_{w \leq 0}$;
        \item if $x \in \C_{w\leq 0}$ and $y \in \C_{w\geq 1}$, then
        \[
            \Hom_\C(x,y) = 0;
        \]
        \item for any $x \in \C$, there exists a cofiber sequence
        \[
            x_{\leq 0} \to x \to x_{\geq 1},
        \]
        where $x_{\leq 0} \in \C_{w\leq 0}$ and $x_{\geq 1} \in \C_{w\geq 1}$.
    \end{enumerate}
\end{defn}

As in the case of $t$-structures, one can define the notion of right bounded, left bounded and bounded weight structures. Since it is the only definition we are going to use, we recall the definition of a bounded weight structure.

\begin{defn}A weight structure $(\C_{w \geq 0}, \C_{w\leq 0})$ on a stable $\infty$-category $\C$ is said to be \emph{bounded} if the inclusion
    \[
        \C^b = \bigcup_{n} \left( \C_{w \geq -n} \cap \C_{w \leq n} \right) \to \C
    \]
    is an equivalence.
\end{defn}

\begin{example}\label{ex:chain_weight}
    For any additive idempotent-complete category $\A$, the stable $\infty$-category of chain complexes $\operatorname{Ch}^b(\A) \coloneqq \mathrm{N}_{\mathrm{dg}}(\K^b(\A))$ is endowed with a weight structure given by
    \[
    \begin{split}
        &\operatorname{Ch}^b(\A)_{w\geq 0} \coloneqq \{ X \in \operatorname{Ch}^b(\A) \mid X_i = 0 \text{ if } i<0\}; \\
        &\operatorname{Ch}^b(\A)_{w\leq 0} \coloneqq \{ X \in \operatorname{Ch}^b(\A) \mid X_i = 0 \text{ if } i>0\}.
    \end{split}
    \]
\end{example}

\begin{defn}
    Let $\C$ be a stable $\infty$-category equipped with a weight structure $(\C_{w \leq 0}, \C_{w \geq 0})$. We define the heart of the weight structure as:
    \[
        \C^{\heartsuit_w} \coloneqq \C_{w \leq 0} \cap \C_{w \geq 0}.
    \]
    If $\mathscr{D}$ is another stable $\infty$-category equipped with a weight structure, then a weigth exact functor $F \colon \C \to \mathscr{D}$ between stable $\infty$-categories is an exact functor which sends the heart of $\C$ to the heart of $\mathscr{D}$.
\end{defn}

The heart of a weight structure is always an additive idempotent-complete $\infty$-category, but rarely an Abelian category\footnote{One can already see this being the case for Example \ref{ex:chain_weight}}. Nevertheless, they are fundamental to the theory, since a (bounded) weight structure is determined by its heart $\C^{\heartsuit_w}$. 

\begin{defn}
    Let $\C$ be a stable $\infty$-category and $\mathscr{N}$ be a collection of objects of $\C$. We say $\mathscr{N}$ is negative (resp. positive, resp. tilting) if for every $x, \, y \in \mathscr{N}$
    \[
        \Hom_\C(x,\Sigma^n y) = 0
    \]
    for every $n > 0$ (resp. $n < 0$, resp. $n \neq 0$).
\end{defn}

The following lemma gives a sufficient (and necessary) condition for a subcategory to determine the heart of a bounded weight structure. It is due to Bondarko in \cite[Theorem 4.3.2.II]{bondarko_original}, and the statement in the $\infty$-categorical language is given in \cite[Remark 2.2.6]{weight_infty}.

\begin{lemma}\label{thm:weight_conditions}
    Let $\C$ be a stable $\infty$-category and suppose we are given a negative collection of objects $\mathscr{N}$ which generates $\C$ under finite limits, finite colimits and retracts. Then the categories
    \[
        \C_{w \leq 0} = \{ \textrm{retracts of finite colimits of } \mathscr{N}\}
    \]
    and
    \[
        \C_{w \geq 0} = \{ \textrm{retracts of finite limits of } \mathscr{N} \}
    \]
    defines a bounded weight structure on $\C$ whose heart is the smallest additive subcategory closed under retracts containing $\mathscr{N}$. 
\end{lemma}

If the collection is tilting, then the heart $\C^{\heartsuit_w}$ is a classical category (in other words, $\C^{\heartsuit_w} \cong \h \C^{\heartsuit_w}$), indeed by definition there is no higher $\Hom$ between the elements of $\N$ and since $\N$ generates the heart under retracts, there is no higher $\Hom$ between elements in the heart. We are going to see that the only such examples are the ones from Example \ref{ex:chain_weight}. In order to prove this, we define the weight complex functor.

\begin{prop}[\cite{sosnilo_heart}, Proposition 3.3] \label{thm:weight_yoneda_2}
Let $\C$ be a stable $\infty$-category equipped with a bounded weight structure. Then:
\begin{enumerate}[(1)]
    \item The functor
    \[
        \C \to \Fun^{\add}((\C^{\heartsuit_w})^{\op}, \Sp)
    \] 
    given by the restriction of the Yoneda functor is exact and fully faithful.
    \item Let $\C'$ be another stable $\infty$-category equipped with a weight structure. Then the restriction functor $\Fun^{w\textrm{-ex}}(\C, \C') \to \Fun^{\add}(\C^{\heartsuit_w}, \C'^{\heartsuit_w})$ is an equivalence of $\infty$-categories.
\end{enumerate}
\end{prop}

\begin{defn}
    Let $\C$ be a stable $\infty$-category equipped with a bounded weight structure. The \emph{weight complex functor} is the weight exact functor $t \colon \C \to \mathrm{Ch}^b(\h\C^{\heartsuit_w})$ being sent to the additive functor $\C^{\heartsuit_w} \to \h\C^{\heartsuit_w}$ under the equivalence of $\infty$-categories given by Proposition \ref{thm:weight_yoneda_2}(2).
\end{defn}

As we are going to see, this construction is both symmetric monoidal, when the weight structure is compatible with the symmetric monoidal structure, and functorial with respect to weight exact functors. 

\begin{prop}[\cite{sosnilo_heart}, Corollary 3.5]\label{prop:funct_weight}
    Let $\C$ and $\C'$ be two stable $\infty$-categories equipped with bounded weight structures and $F \colon \C \to \C'$ be a weight exact functor between them. The following diagram commutes.
    \[
    \begin{tikzcd}[column sep = 40pt]
    \C \ar["t_\C"']{d} \ar["F"]{r} & \C' \ar["t_{\C'}"]{d} \\
    \mathrm{Ch}^b(\h\C^{\heartsuit_w}) \ar["\mathrm{Ch}^b(\h F^{\heartsuit_w})"']{r} & \mathrm{Ch}^b(\h\C'^{\heartsuit_w})
    \end{tikzcd}
    \]
\end{prop}

\begin{defn}
Let $\C^\otimes$ be a stably symmetric monoidal $\infty$-category. We say a weight structure $(\C_{w \geq 0}, \C_{w\leq 0})$ is \emph{compatible with the symmetric monoidal structure} if $\C_{w\geq 0}$ and $\C_{w\leq 0}$ are stable under tensor product.
\end{defn}

\begin{lemma}[\cite{ko_weight_monoidal}, Corollary 4.5]
    Let $\C^\otimes$ be a stably symmetric monoidal $\infty$-category equipped with a bounded weight structure compatible with the symmetric monoidal structure. Then the weight complex functor admits a symmetric monoidal enhancement.
\end{lemma}

\begin{prop} \label{thm:tilting_equivalence_2}
    Let $\C$ be a stable $\infty$-category endowed with a bounded weight structure whose heart $\C^{\heartsuit_w}$ is tilting. Then the weight complex functor is an equivalence:
    \[
        \C \cong \mathrm{Ch}^b(\C^{\heartsuit_w}).
    \]
    If $\C^\otimes$ is a stably symmetric monoidal $\infty$-category and the bounded weight structure is compatible with the symmetric monoidal structure, then this equivalence is symmetric monoidal.
\end{prop}

\begin{proof}
    By Proposition \ref{thm:weight_yoneda_2}(2), one has
    \[
        \Fun^{w\textrm{-ex}}(\C, \Ch^b(\h \C^{\heartsuit_w})) \to \Fun^{\add}(\C^{\heartsuit_w}, \h \C^{\heartsuit_w}),
    \]
    and, since the heart is tilting, $\C^{\heartsuit_w} \cong \h \C^{\heartsuit_w}$, which concludes our proof.
\end{proof}

Observe that if $\C^{\heartsuit_w}$ is tilting, then $\C^{\heartsuit_w} = \h \C^{\heartsuit_w}$ and we recover Example \ref{ex:chain_weight}. Moreover, to verify that the heart $\C^{\heartsuit_w}$ is tilting, it suffices to verify that a generating collection $\mathscr{N}$ is tilting.

\begin{cor}[\cite{motivic_springer}, Corollary 2.16] \label{thm:tilting_algebra_2}
    Let $\C$ be a stable $\infty$-category admiting an auto-equivalence $\langle 1 \rangle$ and $x \in \C$ an object. If the collection $\mathscr{N} = \{x \langle n \rangle \mid n \in \ZZ\}$ is tilting and generates $\C$ under finite limits, finite colimits and retracts, then 
    \[
        \C \cong \mathrm{grPerf}(\End_\C^\bullet(x)),
    \]
    where $\End_\C^\bullet(x) \coloneqq \bigoplus_{n \in \ZZ} \Hom_\C(x, x\langle n \rangle)$ and $\mathrm{grPerf}(A^\bullet)$ is the category of graded perfect complexes of $A^\bullet$-modules. 
\end{cor}

\newcommand{\etalchar}[1]{$^{#1}$}


\begin{thebibliography}{AHJR17b}

\bibitem[AC23]{Chow-weight-stacks}
Dhyan Aranha and Chirantan Chowdhury.
\newblock The chow weight structure for geometric motives of quotient stacks.
\newblock Preprint, {arXiv}:2306.10557 [math.{AG}] (2023), 2023.

\bibitem[AHJR14]{weyl_fourier_springer_integral}
Pramod~N. Achar, Anthony Henderson, Daniel Juteau, and Simon Riche.
\newblock Weyl group actions on the {Springer} sheaf.
\newblock {\em Proc. Lond. Math. Soc. (3)}, 108(6):1501--1528, 2014.

\bibitem[AHJR17a]{modular_springer_ii}
Pramod~N. Achar, Anthony Henderson, Daniel Juteau, and Simon Riche.
\newblock Modular generalized {Springer} correspondence. {II}: {Classical}
  groups.
\newblock {\em J. Eur. Math. Soc. (JEMS)}, 19(4):1013--1070, 2017.

\bibitem[AHJR17b]{modular_springer_iii}
Pramod~N. Achar, Anthony Henderson, Daniel Juteau, and Simon Riche.
\newblock Modular generalized {Springer} correspondence. {III}: {Exceptional}
  groups.
\newblock {\em Math. Ann.}, 369(1-2):247--300, 2017.

\bibitem[Aok20]{ko_weight_monoidal}
Ko~Aoki.
\newblock The weight complex functor is symmetric monoidal.
\newblock {\em Adv. Math.}, 368:9, 2020.
\newblock Id/No 107145.

\bibitem[Bon10]{bondarko_original}
M.~V. Bondarko.
\newblock Weight structures vs. {{\(t\)}}-structures; weight filtrations,
  spectral sequences, and complexes (for motives and in general).
\newblock {\em J. \(K\)-Theory}, 6(3):387--504, 2010.

\bibitem[Car85]{carter_finite}
Roger~W. Carter.
\newblock Finite groups of {Lie} type. {Conjugacy} classes and complex
  characters.
\newblock Pure and {Applied} {Mathematics}. {A} {Wiley}-{Interscience}
  {Publication}. {Chichester}-{New} {York} etc.: {John} {Wiley} and {Sons}.
  {XII}, 544 p. {{\textsterling}} 42.50 (1985)., 1985.

\bibitem[CD24]{non-representable}
Chirantan Chowdhury and Alessandro D'Angelo.
\newblock Non-representable six-functor formalisms.
\newblock Preprint, {arXiv}:2409.20382 [math.{AG}] (2024), 2024.

\bibitem[CGR06]{normalizer_torus}
V.~Chernousov, P.~Gille, and Z.~Reichstein.
\newblock Resolving {{\(G\)}}-torsors by abelian base extensions.
\newblock {\em J. Algebra}, 296(2):561--581, 2006.

\bibitem[DGA{\etalchar{+}}11]{SGA3}
Michel Demazure, Alexander Grothendieck, M~Artin, J.-E. Bertin, P.~Gabriel,
  M.~Raynaud, and J.-P. Serre, editors.
\newblock {\em S{\'e}minaire de g{\'e}om{\'e}trie alg{\'e}brique du {Bois}
  {Marie} 1962-64. {Sch{\'e}mas} en groupes ({SGA} 3). {Tome} {III}:
  {Structure} des sch{\'e}mas en groupes r{\'e}ductifs}, volume~8 of {\em Doc.
  Math. (SMF)}.
\newblock Paris: Soci{\'e}t{\'e} Math{\'e}matique de France, new annotated
  edition of the 1970 original published bei {Springer} edition, 2011.

\bibitem[DH23]{nisnevic_smooth_stacks}
Neeraj Deshmukh and Jack Hall.
\newblock On the {Motivic} {Homotopy} {Type} of {Algebraic} {Stacks} (with an
  appendix joint with {Jack} {Hall}).
\newblock Preprint, {arXiv}:2304.10631 [math.{AG}] (2023), 2023.

\bibitem[DS25]{morel-voevodsky_stacks}
Neeraj Deshmukh and Felix Sefzig.
\newblock The {Morel}-{Voevodsky} {Construction} over {Algebraic} {Stacks}.
\newblock Preprint, {arXiv}:2506.12820 [math.{AG}] (2025), 2025.

\bibitem[Ebe21]{springer_motives}
Jens~Niklas Eberhardt.
\newblock Springer motives.
\newblock {\em Proc. Am. Math. Soc.}, 149(5):1845--1856, 2021.

\bibitem[Ebe24]{k-motives}
Jens~Niklas Eberhardt.
\newblock K-motives, {Springer} {Theory} and the {Local} {Langlands}
  {Correspondence}.
\newblock Preprint, {arXiv}:2401.13052 [math.{RT}] (2024), 2024.

\bibitem[ES22a]{motivic_springer}
Jens~Niklas Eberhardt and Catharina Stroppel.
\newblock Motivic {Springer} theory.
\newblock {\em Indag. Math., New Ser.}, 33(1):190--217, 2022.

\bibitem[ES22b]{weight_infty}
Elden Elmanto and Vladimir Sosnilo.
\newblock On nilpotent extensions of {{\(\infty\)}}-categories and the
  cyclotomic trace.
\newblock {\em Int. Math. Res. Not.}, 2022(21):16569--16633, 2022.

\bibitem[ES23]{reduced_motives}
Jens~Niklas Eberhardt and Jakob Scholbach.
\newblock Integral motivic sheaves and geometric representation theory.
\newblock {\em Adv. Math.}, 412:42, 2023.
\newblock Id/No 108811.

\bibitem[HM24]{six_functor_formalism}
Claudius Heyer and Lucas Mann.
\newblock 6-{Functor} {Formalisms} and {Smooth} {Representations}.
\newblock Preprint, {arXiv}:2410.13038 [math.{CT}], 2024.

\bibitem[Hum92]{humphreys_coxeter}
James~E. Humphreys.
\newblock {\em Reflection groups and {Coxeter} groups}, volume~29 of {\em Camb.
  Stud. Adv. Math.}
\newblock Cambridge: Cambridge University Press, 1992.

\bibitem[Jin16]{jin_borel-moore}
Fangzhou Jin.
\newblock Borel-{Moore} motivic homology and weight structure on mixed motives.
\newblock {\em Math. Z.}, 283(3-4):1149--1183, 2016.

\bibitem[Kha25]{khan-lisse}
Adeel~A. Khan.
\newblock Lisse extensions of weaves.
\newblock Preprint, {arXiv}:2501.04114 [math.{AG}] (2025), 2025.

\bibitem[KL09]{KL}
Mikhail Khovanov and Aaron~D. Lauda.
\newblock A diagrammatic approach to categorification of quantum groups. {I}.
\newblock {\em Represent. Theory}, 13:309--347, 2009.

\bibitem[Kri13]{krishna_chow}
Amalendu Krishna.
\newblock Higher {Chow} groups of varieties with group action.
\newblock {\em Algebra Number Theory}, 7(2):449--506, 2013.

\bibitem[Lus84]{int_coh}
G.~Lusztig.
\newblock Intersection cohomology complexes on a reductive group.
\newblock {\em Invent. Math.}, 75:205--272, 1984.

\bibitem[Rid13]{rider_springer}
Laura Rider.
\newblock Formality for the nilpotent cone and a derived {Springer}
  correspondence.
\newblock {\em Adv. Math.}, 235:208--236, 2013.

\bibitem[Rou08]{rouquier}
Rapha{\"e}l Rouquier.
\newblock 2-{Kac}-{Moody} algebras.
\newblock Preprint, {arXiv}:0812.5023 [math.{RT}] (2008), 2008.

\bibitem[Sch12]{lectures_schiffman}
Olivier Schiffmann.
\newblock Lectures on {Hall} algebras.
\newblock In {\em Geometric methods in representation theory. II. Selected
  papers based on the presentations at the summer school, Grenoble, France,
  June 16 -- July 4, 2008}, pages 1--141. Paris: Soci{\'e}t{\'e}
  Math{\'e}matique de France, 2012.

\bibitem[Sos19]{sosnilo_heart}
Vladimir Sosnilo.
\newblock Theorem of the heart in negative {{\(K\)}}-theory for weight
  structures.
\newblock {\em Doc. Math.}, 24:2137--2158, 2019.

\bibitem[Spi18]{spitzweck}
Markus Spitzweck.
\newblock {\em A commutative {{\({\mathbb P}^1\)}}-spectrum representing
  motivic cohomology over {Dedekind} domains}, volume 157 of {\em M{\'e}m. Soc.
  Math. Fr., Nouv. S{\'e}r.}
\newblock Soci{\'e}t{\'e} Math{\'e}matique de France (SMF), Paris, 2018.

\bibitem[Spr76]{springer_1}
T.~A. Springer.
\newblock Trigonometric sums, {Green} functions of finite groups and
  representations of {Weyl} groups.
\newblock {\em Invent. Math.}, 36:173--207, 1976.

\bibitem[Spr78]{springer_2}
T.~A. Springer.
\newblock A construction of representations of {Weyl} groups.
\newblock {\em Invent. Math.}, 44:279--293, 1978.

\bibitem[SS70]{conjugacy_classes}
T.~A. Springer and R.~Steinberg.
\newblock Conjugacy classes.
\newblock Sem. algebr. {Groups} related finite {Groups} {Princeton} 1968/69,
  {Lect}. {Notes} {Math}. 131, {E1}-{E100} (1970)., 1970.

\bibitem[ST18]{jacobson-morozov}
David~I. Stewart and Adam~R. Thomas.
\newblock The {Jacobson}-{Morozov} theorem and complete reducibility of {Lie}
  subalgebras.
\newblock {\em Proc. Lond. Math. Soc. (3)}, 116(1):68--100, 2018.

\bibitem[Ste75]{torsion_steinberg}
Robert Steinberg.
\newblock Torsion in reductive groups.
\newblock {\em Adv. Math.}, 15:63--92, 1975.

\bibitem[SVW18]{eq_mot}
Wolfgang Soergel, Rahbar Virk, and Matthias Wendt.
\newblock Equivariant motives and geometric representation theory. (with an
  appendix by {F}. {H{\"o}rmann} and {M}. {Wendt}).
\newblock Preprint, {arXiv}:1809.05480 [math.{RT}], 2018.

\bibitem[SW11]{stroppel-webster}
Catharina Stroppel and Ben Webster.
\newblock Quiver {Schur} algebras and q-{Fock} space.
\newblock Preprint, {arXiv}:1110.1115 [math.{RA}] (2011), 2011.

\bibitem[SW18]{perverse_category_o}
Wolfgang Soergel and Matthias Wendt.
\newblock Perverse motives and graded derived category {{\({\mathcal{O}}\)}}.
\newblock {\em J. Inst. Math. Jussieu}, 17(2):347--395, 2018.

\bibitem[VV11]{klr-algebras_computation}
M.~Varagnolo and E.~Vasserot.
\newblock Canonical bases and {KLR}-algebras.
\newblock {\em J. Reine Angew. Math.}, 659:67--100, 2011.

\bibitem[Zho22]{affine_paving_quivers}
Xiaoxiang Zhou.
\newblock Affine {Pavings} of {Quiver} {Flag} {Varieties}.
\newblock Preprint, {arXiv}:2206.00444 [math.{RT}] (2022), 2022.

\end{thebibliography}
\end{document}